\documentclass[reqno,10pt]{article}
\usepackage[cp1250]{inputenc}
\usepackage{graphicx}
\usepackage{amsmath}
\usepackage{amssymb}
\usepackage{amsthm}
\usepackage{enumitem}
\usepackage{bbm}
\usepackage[l2tabu,orthodox]{nag}
\usepackage[all,error]{onlyamsmath}
\usepackage[strict]{csquotes}
\usepackage{url}
\usepackage[english]{babel}
\usepackage[T1]{fontenc}
\usepackage{polski}
\newtheorem{theorem}{Theorem}
\newtheorem{lemma}[theorem]{Lemma}

\theoremstyle{definition}

\newtheorem{example}{Example}[section]

\numberwithin{equation}{section}
\numberwithin{theorem}{section}

\newenvironment{OMabstract}{\noindent\textbf{Abstract.} }{\medskip}
\newenvironment{OMsubjclass}{\noindent\textbf{Mathematics Subject Classification (2020):} }{\medskip}
\newenvironment{OMkeywords}{\noindent\textbf{Keywords:}  }{\medskip}

\def\>{\rangle}
\newcommand{\<}{\langle}

\newcommand{\bP}{{P}}
\newcommand{\bbP}{\mathbb{P}}

\newcommand{\T}{\mathbb{T}}
\newcommand{\C}{\mathbb{C}}
\newcommand{\N}{\mathbb{N}}
\newcommand{\Z}{\mathbb{Z}}

\newcommand{\bA}{{A}}

\newcommand{\bI}{{I}}

\newcommand{\bU}{{U}}
\newcommand{\bS}{{S}}
\newcommand{\bfI}{\mathbf{I}}
\newcommand{\bfA}{\mathbf{A}}
\newcommand{\bF}{{F}}

\newcommand{\bfO}{\mathbf{O}}
\newcommand{\bfP}{\mathbf{P}}

\newcommand{\cB}{\mathcal{B}}

\newcommand{\E}{\mathcal{E}}

\newcommand{\cH}{\mathcal{H}}

\newcommand{\cS}{\mathcal{S}}
\newcommand{\cP}{\mathcal{P}}

\newcommand{\X}{\mathcal{X}}

\newcommand{\pc}{c(\mathbb{Z};\mathbb{C}^d)}

\newcommand{\pco}{c_0(\mathbb{Z};\mathbb{C}^d)}

\newcommand{\elj}{l_1(\mathbb{Z};\mathbb{C}^d)}
\newcommand{\eljm}{l_1(\mathbb{Z};M_d(\C))}

\newcommand{\elpm}{l_p(\mathbb{Z};M_d(\C))}
\newcommand{\eld}{l_2(\mathbb{Z};\mathbb{C}^d)}

\newcommand{\elp}{l_p(\mathbb{Z};\mathbb{C}^d)}
\newcommand{\elpd}{l_p(\mathbb{Z};\mathbb{C}^2)}
\newcommand{\elpj}{l_p(\mathbb{Z};\mathbb{C})}
\newcommand{\elpp}{l_{p'}(\mathbb{Z};\mathbb{C}^d)}
\newcommand{\eli}{l_{\infty}(\mathbb{Z};\mathbb{C}^d)}

\newcommand{\EL}{L_2(\mathbb{T};\mathbb{C}^d)}
\newcommand{\ELj}{L_2(\mathbb{T})}

\newcommand{\ELDM}{L_2(\mathbb{T}; M_d(\C))}
\newcommand{\ELPM}{L_p(\mathbb{T}; M_d(\C))}

\newcommand{\ELJM}{L_1(\mathbb{T}; M_d(\C))}

\newcommand{\ELp}{L_p(\mathbb{T})}

\newcommand{\esssup}{\text{ess\,sup}\,}
\newcommand{\tr}{\text{tr} \,}

\begin{document}

\author{Ewelina Zalot} 
\title{Spectral resolutions for non-self-adjoint block convolution operators}
\date{}
\maketitle


\begin{OMabstract}
	The paper concerns the spectral theory for a class of non-self-adjoint block convolution operators. We mainly discuss the spectral representations of such operators. It is considered the general case of operators defined on Banach spaces. The main results are applied to periodic Jacobi matrices.
\end{OMabstract}

\begin{OMkeywords}
    spectral operators, chains, triangular decomposition, Laurent ope\-rators, Jacobi matrices.
\end{OMkeywords}

\begin{OMsubjclass}
47B40, 47B28, 47B36, 47B35, 47B39.
\end{OMsubjclass}


\section{Introduction}  

In this work we deal with the spectral theory of a certain class of non-self-adjoint convolution operators in Banach spaces. One of the most significant approaches to spectral analysis of non-self-adjoint operators was introduced by N. Dunford, who defined a spectral operator and studied its spectral properties (\cite{dun2}, see also \cite {dun0, dun1}). The generalization of spectral resolutions in the sense of N.~Dunford was the concept of chains of projections, which was transferred from the finite dimensional case to infinitely dimensional spaces. The first triangular representations for a class of non-self-adjoint bounded operators were introduced and developed by M.S. Livshits \cite{Livsic_1}. In turn, M.S. Brodski\v{i}  (\cite{Brodskii} see also \cite{BrodskiiGokKr}) dealt with triangular representations for non-self-adjoint operators with respect to spectral chains. In monographs I.~Gohberg and M. Krein \cite{gohkrein_1} and I.~Gohberg, S.~Goldberg and M.A.~Kaashoek \cite{gohgold} described, among others, resolutions with respect to spectral chains for Volterra operators.

In this paper we give a spectral decomposition of block non-self-adjoint convolution operators, generalizing the classical Schur theory on triangular decomposition of matrices. We also present a construction of invariant chains under the convolution non-self-adjoint operator $\bA$ acting in the Banach space, generated by the block matrix of the form
 \begin{equation*}
 	\left[ \begin{array}{ccccc}
 		\ddots & \vdots & \vdots & \vdots & \\
 		\dots & A_{0} & A_{-1} & A_{-2} &  \dots  \\
 		\dots & A_{1} & A_{0} & A_{-1} & \dots  \\
 		\dots & A_{2} & A_{1} & A_{0} & \dots  \\
 		& \vdots & \vdots & \vdots & \ddots
 	\end{array} \right ],
 \end{equation*}
 \noindent  where $A_j$ $(j = 0, \pm 1, \pm 2, \dots )$ are finite matrices of the same dimensions. \linebreak In this way we shall generalize the results presented in the paper \cite{naiman} containing spectral resolutions of non-self-adjoint tridiagonal periodic Jacobi matrices \linebreak (see also \cite{naiman-1, naiman-2}  or \cite[Chapter 14]{Glaz}) and in the paper \cite{zalot}. A motivation  for our considerations comes from \cite{cojuhari} as well as \cite{cojuhari2}. We also show the applications of our results to Jacobi type matrices, which appear in many problems in mathema\-tics, physics, mathematical physics, etc. As a particular case it can be mentioned the Born - von Karman model of the crystal lattice. It is also worth mentioning that periodic Jacobi matrices can be treated as a finite-difference analog of the one-dimensional Hill's differential operator (\cite{MagnusWinkler}, see also \cite{Glaz}, whose study is based on Floquet's theory). Jacobi-type matrices and their associated continuous fractions were for a long time object of the study for various authors (see \cite{akhiezer}, \cite{Geronimus} and the bibliography given there). 
		
The paper is organized as follows. In Section 2 we recall some key definitions and properties of chains in Banach spaces from \cite{gohgold} and \cite{gohkrein_1}. Our main results are found in Section 3, where we shall give a spectral decomposition of block non-self-adjoint convolution operator with respect to a chain in Banach space. Then we give an alternative method of constructing an invariant chain under the operator $ \bA $ acting in a Hilbert space. In Section 4, certain applications of the main results to  Jacobi type matrices are given. Some concrete examples important by themselves are also presented.

Spectral resolutions with respect to spectral measures and conditions under which the block non-self-adjoint convolution operator is a spectral or scalar operator will be described in the next paper.

\section{Preliminaries. Chains and triangular decompositions}


In this section we outline those aspects and results of general theory of chains in Banach spaces which will be used later on in basic text of our work. \linebreak The material presented is taken mainly from \cite{gohgold} and \cite{gohkrein_1}.

Throughout this section $\X$ is a Banach space. $\cB (\X)$ denotes the algebra of all bounded operators on $\X$ to itself. By a \textit{projection} of $\X$ we mean an idempotent operator $P$ in $\cB (\X)$, i.e., an operator $P \in \cB (\X)$ for which $P^2=P$. In the set $\cP(\X)$ of all projections on $\X$ it can be entered a \textit{partial ordering} by setting  $P_1 \le P_2$ if and only if $P_1P_2 = P_2P_1 = P_1.$ We write $P_1 < P_2$ in case $P_1 \le P_2$ and $P_1 \neq P_2$. For any projection $P \in \cP(\X)$ there holds $O \le P \le I$, where $O$ and $I$ denote the zero and identity projections on $\X$, respectively.

A set $\bbP$ of projections  $P \in \cP(\X)$ is called a \textit{chain} on $\X$ if $O, I \in \bbP$ and $\bbP$ is linearly ordered with respect to the ordering $\le$ inherited from $\cP(\X)$. If $\bbP$ is a~chain on $\X$, then the set $\bbP^c$ of all $P$ such that $I-P \in \bbP$ is also a chain on $\X$. $\bbP^c$ is called \textit{complement chain} of $\bbP$.

A chain $\bbP$ is said to be \textit{maximal} if $\bbP$ is not properly included in any other chain (on $\X$). As follows from Zorn's Lemma, every chain is contained in some maximal chain. A pair $(P_1, P_2)$, where $P_1, P_2 \in \bbP$, is said to be a \textit{jump} (or \textit{discontinuity}) of  $\bbP$ if $P_1 < P_2$ and for any $P \in \bbP$ either $P \le P_1$ or $P \ge P_2$. The rank of the projection $P_2-P_1$ is called the \textit{dimension} of the jump $(P_1, P_2)$. Chains without any jumps are said to be \textit{continuous}. The jumps of a maximal chain can be only one-dimensional.

A chain $\bbP$ on $\X$ is said to be \textit{invariant} under an operator $A$ on $\cB(\X)$ if $PAP = AP$ for each $P \in \bbP$. It is said that $A$ is reduced by $\bbP$, or $\bbP$ is \textit{reducing} for $A$, if both $\bbP$ and $\bbP^c$ are invariant under $A$, in other words, if $AP = PA$ for $P \in \bbP.$

The approaches undertaken in achieving our main goals essentially involve triangular representations similar to those due to Schur, well known for the finite dimensional case. In this context, given a (bounded) operator defined on a Banach space $\X$, we have in mind the possibility of additive lower-upper triangular decomposition
\begin{equation}\label{LUdecomp}
	A = A_- + A_0 + A_+,
\end{equation}
with respect to an appropriate chain $\bbP$. In \eqref{LUdecomp}, $A_0$ is the \textit{diagonal} of the ope\-ra\-tor $A$ (under the assumption that it exists) and is expressed by the following integral
\begin{equation}\label{1.2:repr_diag_calk_lan}
	A_0 = \int_{\bbP} (dP)A(dP)
\end{equation}
along the chain $\bbP$, whereas $A_-$ and $A_+$ are lower and upper block triangular representations with respect to $\bbP$, respectively. Moreover, $A_-$ and $A_+$ are expressed by
\begin{equation}\label{1.2:repr_LU_calk_lan}
	A_- = \int_{\bbP} (dP)AP, \quad A_+ = \int_{\bbP} PA(dP).
\end{equation}
For more information on integration along chains see \cite{BrodskiiGokKr} and \cite{gohkrein_1} (see also \cite{gohgold}).

In particular if the chain $\bbP$ is invariant under the operator $A$, then
\begin{equation*}
	A = A_0 + A_+,
\end{equation*}
and $A = A_0$, i.e.,
\begin{equation*}
		A = \int_{\bbP} (dP)A(dP)
\end{equation*}
in case that $\bbP$ is reducing for $A$.

\section{Spectral resolutions}

In this section, we present the main result of the paper. It concerns spectral decomposition of the considered convolution operators, summing up the extension to them of the classical Schur theory on triangular decomposition of matrices.

\subsection{Block Laurent operators. Matrix symbols}

Throughout this section $M_{d}(\C)$ stands for the set of all $d \times d$ complex matrices ($d \in \mathbb{N})$. The usual norms in all spaces $\C^{d}$ and $M_{d}(\C)$ are denoted by the same symbol $| \cdot |$. By $\elp$ $(1 \le p \le \infty)$ we mean the Banach space of all $\C^d$-valued sequences $u=(u_n)$, $u_n \in \C^d$ ($n \in \Z$), such that \begin{equation*}
	\|u\|_{\elp}:= \left(\sum_{n \in \Z} |u_n|^p \right)^{\frac{1}{p}}<\infty, \quad 1 \le p < \infty,
\end{equation*} 

\noindent and
\begin{equation*}
	\|u\|_{\infty} := \sup_{n \in \Z} |u_n| < \infty, \quad p = \infty.
	\end{equation*} 

\noindent Similarly to the above spaces, we define spaces $\elpm$ $(1 \le p \le \infty)$, \linebreak in which instead of $\C^d$ we consider $M_{d}(\C)$.

$\ELp$ is a Banach space of all measurable functions $u$ such that
\begin{equation*}
\|u\|_p^p := \int_{\T} |u(z)|^p d \tau < \infty, \quad 1 \le p < \infty,
\end{equation*}
and
\begin{equation*}
\|u\|_{\infty} := \esssup_{z \in \T} |u(z)| < \infty, \quad p = \infty,
\end{equation*}
where $\T$ is a unit circle in $\C$
\begin{equation*}
\T := \{z \in \C: |z| = 1\},
\end{equation*}
and $\tau$ is a normalized Lebesgue measure 
\begin{equation*}
d \tau := \frac{1}{2 \pi i} \frac{d z}{z}.
\end{equation*}
$\ELPM$ is a space of all $d \times d$ measurable matrix-valued functions with entries belonging to the space $\ELp$.

Let $\X$ be one of the sequence spaces: $\elp \, (1 \le p < \infty)$, $\pco$, $\pc$, $\eli$, and let $A = (A_n)$ be a sequence belonging to $\eljm$, i.e.,
\begin{equation}\label{2.1:war_zb}
	\|A\| = \sum \limits_{n = - \infty}^{\infty} |A_n| < \infty.
\end{equation}
We define the operator $\bA: \X \to \X$ as follows
\begin{equation}\label{2.1:op_splot}
	\bA u = A*u = \bigg(\sum_{k=-\infty}^{\infty} A_{n-k} u_k \bigg)_{n=-\infty}^{\infty}, \quad u = (u_k) \in \X,
\end{equation}
\noindent where ``$\ast$'' denotes a convolution operation. $\bA$ is called a \textit{convolution operator}, and it is represented by the matrix
\begin{equation}\label{2.1:mac_laurent}
	\left[ \begin{array}{ccccc}
		\ddots & \vdots & \vdots & \vdots & \\
		\dots & A_{0} & A_{-1} & A_{-2} &  \dots  \\
		\dots & A_{1} & A_{0} & A_{-1} & \dots  \\
		\dots & A_{2} & A_{1} & A_{0} & \dots  \\
		& \vdots & \vdots & \vdots & \ddots
	\end{array} \right ],
\end{equation}
which entries are matrices given by $A_{jk} = A_{j-k}$ $(j, k = 0, \pm 1, \pm 2, \dots)$. \linebreak The matrix given by  \eqref{2.1:mac_laurent} is  also called the {\it block Laurent operator}. As is easily seen, the operator $\bA$ can be formally expanded into the series
\begin{equation}\label{2.1:splot_przes}
	\bA  = \sum_{n=-\infty}^{\infty} A_{n} \bS^n,
\end{equation}
where $\bS$ is the (block) shift operator
\begin{equation*}
	\bS u = (u_{n-1}), \quad u = (u_n)\in \X,
\end{equation*}

\noindent The series in \eqref{2.1:splot_przes}, due to \eqref{2.1:war_zb}, is convergent in the operator norm. 

In this way the operator $\bA$ can be regarded as the value of the matrix-valued function
\begin{equation} \label{2.1:def_mac}
	A(z) = \sum \limits_{n = - \infty}^{\infty} A_n z^n, \quad z \in \mathbb{T},
\end{equation}

\noindent  of the operator $\bS$, i.e.,
\begin{equation*} 
	\bA = A(\bS).
\end{equation*}

The matrix-valued function $A(z)$ is called the {\textit{symbol}} of the operator $\bA$. Note that if  the convergent condition \eqref{2.1:war_zb} is fulfilled then $A(z),$ $z \in \T,$ defined by \eqref{2.1:def_mac} is the sum of an absolutely convergent Fourier series, the coefficients of which are expressed by the formulas
\begin{equation}\label{2.1:wspl_Fourier_znorm}
	A_n = \int_\T A(z) z^{-n} d \tau, \quad n \in \Z.
\end{equation}

\noindent Letting $z = e^{i \varphi},$ $0 \le \varphi < 2\pi$, the expansions \eqref{2.1:def_mac} is writing
\begin{equation*}
	A(e^{i \varphi}) = \sum \limits_{n = - \infty}^{\infty} A_n e^{i n \varphi}, \quad 0 \le \varphi < 2\pi,
\end{equation*}

\noindent where
\begin{equation*}\label{2.1:wspl_Fourier}
	A_n = \frac{1}{2\pi} \int_0^{2\pi} A(e^{i \varphi}) e^{-i n \varphi} d \varphi, \quad n \in \Z.
\end{equation*}

Note that each convolution operator $\bA$ satisfying the condition \eqref{2.1:war_zb} corresponds to the $A(z)$ defined by \eqref{2.1:def_mac} with norm \eqref{2.1:war_zb} and vice versa. The paper also includes bounded convolution operators with the symbol that does not satisfy the condition \eqref{2.1:war_zb}. So we define a class of all matrix-valued functions $A(\cdot)$, which are connected with bounded operators. Let $\X$ be one of the spaces $\elp $ $(1 \le p \le \infty)$. 

Let us define the following class of matrix-valued functions   
\begin{equation*}
		\cS^{d }_{p}(\T) := \{A(\cdot) \in \ELJM: \ \bA \in \cB(\X) \}.
\end{equation*} 
We are in a position to prove one of our main results.
\begin{theorem} \label{2.1_tw_przestrzenie}
	The class of matrix-valued functions $\cS^{d}_{p}(\T)$ $(1 \le p \le \infty)$ is~contained in space $\ELDM$.
\end{theorem}
\vskip.3cm
\begin{proof} 
	
	Let $A(\cdot) \in \cS^{d}_{p}(\T)$. Then the convolution operator $\bA$ defined by \eqref{2.1:op_splot}, where $A_n$ are expressed by \eqref{2.1:wspl_Fourier_znorm}, is bounded, i.e., $\bA \in \cB(\elp)$. It follows that for each  $u =~(u_n) \in \elp$ ($1 \le p <\infty$)	
	\begin{equation*}
		\bA u \in \elp.
	\end{equation*} 
	
	\noindent Let	
	\begin{equation*}
		\tilde u = (\delta_{0n}u_0) = (\dots, 0, u_0, 0, \dots).
	\end{equation*}
	
	\noindent Therefore	
	\begin{equation*}
		\bA \tilde u = \bA * \tilde u = \left(\sum_{k=-\infty}^{\infty} A_{n-k}\, \delta_{0k} u_0 \right) = (A_nu_0).
	\end{equation*}
	
	\noindent This implies that	
	\begin{equation*}
		(A_nu_0) \in \elp 
	\end{equation*}
	
	\noindent for each $u_0 \in \C^d$.

	Assume that $1 \le p \le 2$. Then $\elp \subset \eld$, which entails that	
	\begin{equation*}
		(A_nu_0) \in \eld, \quad u_0 \in \C^d. 
	\end{equation*}
	
	\noindent As a consequence,	
	\begin{equation}\label{2.1:zb_u0}
		\sum_{n=-\infty}^{\infty} |A_nu_0|^2 < \infty, \quad u_0 \in \C^d. 
	\end{equation}
	
	\noindent Since $u_0$ was chosen arbitrarily, the estimation \eqref{2.1:zb_u0} holds for canonical vectors $e_j$ ($j = 1, \dots, d$) from the space $\C^d$. Thus we can deduce that	
	\begin{equation*}
		\sum_{n=-\infty}^{\infty} |A_n|^2 = \sum_{n=-\infty}^{\infty} \tr (A_n A_n^*)< \infty,
	\end{equation*}
	
	\noindent so	
	\begin{equation*}
		A(\cdot) \in \ELDM.
	\end{equation*}

	Now let $p \ge 2$. Since $\bA \in \cB(\elp)$, then	
	\begin{equation*}
		\bA^* \in \cB(\elpp), \quad \frac{1}{p} + \frac{1}{p'} = 1,	
	\end{equation*}
	
	\noindent and	
	\begin{equation*}
		\left(A^*_n u_0\right) \in \elpp, \quad u_0 \in \C^d.
	\end{equation*}

	\noindent From the assumption $p \ge 2$ it follows that $1 < p' \le 2$. As a result, we get	
	\begin{equation*}
		\sum_{n=-\infty}^{\infty} |A_n|^2 = \sum_{n=-\infty}^{\infty} \tr (A_n A_n^*) = \sum_{n=-\infty}^{\infty} \tr (A_n^* A_n) = \sum_{n=-\infty}^{\infty} |A^*_n|^2 < \infty.
	\end{equation*}
	
	\noindent Then	
	\begin{equation*}
		A(\cdot) \in \ELDM,
	\end{equation*}
	
	\noindent which finishes the proof.

\end{proof}

\subsection{Schur resolution} \label{sub_schur}

Let $\X$ be one of the following spaces: $\elp$ $(1 \le p \le \infty)$, $\pco$, $\pc$, and consider Laurent operator $\bA: \X \to \X$ defined by \eqref{2.1:op_splot} for which the convergence condition \eqref{2.1:war_zb} is satisfied. Recall that the matrix-valued function $A(z),$ $z \in \T$, defined by \eqref{2.1:def_mac} is called the symbol of the operator $\bA$.
We shall describe a construction of an invariant chain under $\bA$ with respect to which $\bA$ admits a block triangular representation. First we shall obtain a Schur triangular representation for the matrix $A(z)$ in each point $z \in \T$ and then apply the obtained result to the convolution operator $\bA$.

Let  $z \in \T$. By Schur's theorem  (\cite[p. 79]{horn}) the operator $A(z)$ in $\C^d$ is unitarily equivalent to an upper triangular matrix, which we denote by $T(z)$. Namely, there exists a unitary matrix  $U(z)$ such that
\begin{equation} \label{2.4:unit_rown_T}
	A(z) = U(z) T(z) (U(z))^*,
\end{equation}
\noindent where
\begin{equation*} \label{2.4:tacierz}
	T(z) = \left[ \begin{array}{cccc}
		\lambda_1(z) & \mu_{21}(z) & \cdots &  \mu_{d1}(z)   \\
		0 & \lambda_2(z) & \cdots & \mu_{d2}(z)    \\
		\vdots & \vdots &  \ddots & \vdots    \\
		0 & 0 & \cdots &  \lambda_d(z)     \\
	\end{array} \right ]
\end{equation*}
and
\begin{equation*} \label{2.4:ucierz}
	U(z) = \left[ \begin{array}{cccc}
		\varphi^{(1)}(z),& \varphi^{(2)}(z), & \dots, & \varphi^{(d)}(z) \\
	\end{array} \right ] = \left[ \begin{array}{ccc}
		\varphi^{(1)}_1(z)  & \cdots &  \varphi^{(d)}_1(z)   \\
		\varphi^{(1)}_2(z) & \cdots & \varphi^{(d)}_2(z)    \\
		\vdots &   \cdots & \vdots    \\
		\varphi^{(1)}_d(z) &  \cdots &  \varphi^{(d)}_d(z)     \\
	\end{array} \right ].
\end{equation*}
$(U(z))^*$ is the Hermitian adjoint of the matrix $U(z)$, so
\begin{equation*}
	(U(z))^* = (U(z))^{-1}.	
\end{equation*}
The diagonal entries  $\lambda_1(z), \dots, \lambda_d (z)$ of $T(z)$ are the eigenvalues in any prescribed order of the matrix $A(z)$ viewed as an operator in $\C^d$. The orthogonal vectors $\varphi^{(1)}(z), \ldots,$ $\varphi^{(d)}(z) \in \C^d$ form the \textit{Schur orthonormal basis} in $\C^d$, so that
\begin{align*}
	A(z) \varphi^{(1)}(z) & = \lambda_1(z) \varphi^{(1)}(z), \\
	A(z) \varphi^{(k)}(z)  & = \lambda_k(z) \varphi^{(k)}(z) + \sum \limits_{l = 1}^{k - 1} \mu_{kl}(z) \varphi^{(l)}(z), \quad k = 2, \dots, d, \label{2.4:mac_schur}
\end{align*}
 where  $\varphi^{(1)}(z)$ is the eigenvector corresponding to  $\lambda_1(z)$ and $\mu_{kl}(z) \in \C$ \linebreak ($k =2, \ldots, d$, $l = 1, \ldots, d-1$). Note that all the elements $\lambda_l$, $\varphi^{(l)}$ and $\mu_{kl}$ can be selected to be continuous functions of a variable $z \in \T$ (see \cite{ostrowski} and \cite{rellich}).

Denote by
\begin{equation} \label{2.4:skok}
	\Delta P_k (z) := \<\cdot, \varphi^{(k)}(z)\> \varphi^{(k)}(z), \quad k=1, \dots, d,
\end{equation}
\noindent projections on the one-dimensional  subspaces of $\C^d$  generated by the Schur vectors $\varphi^{(k)}(z)$ $(k=1, \dots, d)$, respectively.  $\Delta P_k (z)$ ($k=1, \dots, d$) are orthogonal projections, i.e.,
\begin{equation*}
	\Delta P_k (z) = (\Delta P_k (z))^2 = (\Delta P_k (z))^*, \quad k = 1, \dots, d,
\end{equation*}
and
\begin{equation*}
	\Delta P_k (z)\Delta P_j (z) = 0, \quad k \neq j; \ k, j = 1, \dots, d.
\end{equation*}
Moreover,
\begin{equation*}
	\sum _{k=1}^d \Delta P_l (z) = I,
\end{equation*}
and each of the operators
\begin{equation} \label{2.4:proj}
	P_k(z) = \sum\limits_{l=1}^k \Delta P_l(z), \quad k=1, \dots, d,
\end{equation}

\noindent is an orthogonal projection on  $\C^d$. Projections $P_k(z)$ $(k = 1, \dots, d)$ satisfy the monotonicity condition
\begin{equation}\label{2.4:monot}
	P_k(z)P_n(z) =   \sum\limits_{l=1}^k \sum\limits_{j=1}^n \Delta P_l(z) \Delta P_j(z) = \sum\limits_{l=1}^k \Delta P_l(z) = P_k(z),
\end{equation}

\noindent for $k < n$ $(k, n = 1, \dots, d)$. In view of \eqref{2.4:skok}, \eqref{2.4:proj} and \eqref{2.4:monot} we see that the set
\begin{equation}\label{2.4:lan_skoncz}
	\pi(z): \quad O = P_0(z) \ < \ P_1(z) \ < \ \dots \ < P_k(z) \ < \ \dots \ < \ P_d(z) = I
\end{equation}
\noindent is a finite chain on $\C^d$ of orthogonal projections $P_k(z)$ $(k = 0, \dots, d)$.  $\pi(z)$ is a~maximal chain on $\C^d$, and the pairs $(P_{k-1}(z), P_k(z))$ are one-dimensional jumps of it.

It is easy to show that the chain $\pi(z)$ is invariant under $A(z)$ for each $z$ in~$\T$, i.e., 
\begin{equation}\label{2.4:niezm_mac}
	P_k(z)A(z) P_k(z) = A(z) P_k(z), \quad z \in \T.
\end{equation}

Since the chain $\pi(z)$ is invariant under $A(z)$, then $A(z)$ has the following decomposition
\begin{equation} \label{2.4:rozklad_mac}
	A(z) = A_0(z) + A_+(z),
\end{equation}
\noindent where $A_0(z)$ is the diagonal of $A(z)$ with respect to the chain $\pi(z)$. Since
\begin{align*}
	A_0(z) & = \sum\limits_{l=1}^{d} \Delta P_l(z) A(z) \Delta P_l(z) =  \sum\limits_{l=1}^{d} \Delta P_l(z) \<\cdot, \varphi^{(l)}(z)\> A(z)\varphi^{(l)}(z)\\
	&  = \lambda_l(z) (\Delta P_1(z))^2 + \sum\limits_{l=2}^{d} \Delta P_l(z) \<\cdot, \varphi^{(l)}(z)\> \bigg( \lambda_l(z) \varphi^{(l)}(z) \\
	& \quad + \sum \limits_{j = 1}^{l - 1} \mu_{lj}(z) \varphi^{(j)}(z)\bigg)\\
	& = \sum\limits_{l=1}^{d} \lambda_l(z) \Delta P_l(z) + \sum\limits_{l=2}^{d} \sum \limits_{j = 1}^{l - 1} \<\cdot, \varphi^{(l)}(z)\> \mu_{lj}(z) \<\varphi^{(j)}(z), \varphi^{(l)}(z) \> \varphi^{(l)}(z),
\end{align*}
\noindent it follows that
\begin{equation}\label{2.4:rozklad_mac_A0}
	A_0(z) = \sum\limits_{l=1}^{d} \lambda_l(z)\Delta P_l(z).
\end{equation}

\noindent In turn, the upper triangular matrix $A_+(z)$ is expressed as follows
\begin{align*}\nonumber
	A_+(z) = &  \sum\limits_{l=1}^{d} P_{l-1}(z) A(z) \Delta P_l(z) \\ \nonumber
	= & \sum\limits_{l=2}^{d} \sum\limits_{s=1}^{l-1} \Delta P_s(z) \<\cdot, \varphi^{(l)}(z)\> \bigg( \lambda_l(z) \varphi^{(l)}(z) + \sum \limits_{j = 1}^{l - 1} \mu_{lj}(z) \varphi^{(j)}(z)\bigg)\\ \nonumber
	= &  \sum\limits_{l=2}^{d} \sum\limits_{s=1}^{l-1} \lambda_l(z) \Delta P_s(z)\Delta P_l(z) \\	
	& + \sum\limits_{l=2}^{d} \sum\limits_{s=1}^{l-1} \sum \limits_{j = 1}^{l - 1} \mu_{lj}(z) \<\cdot, \varphi^{(l)}(z)\> \<  \varphi^{(j)}(z), \varphi^{(s)}(z) \> \varphi^{(s)}(z),
\end{align*}
i.e.,
\begin{equation}\label{2.4:rozklad_mac_A+}
	A_+(z) = \sum\limits_{l=2}^{d} \sum\limits_{s=1}^{l-1} \mu_{ls}(z) \<\cdot, \varphi^{(l)}(z)\> \varphi^{(s)}(z).
\end{equation}

\noindent As is seen, $\mu_{ls}(z)$ $(s=1, \dots, l-1; l = 2, \dots, d)$ are elements  of the matrix $A_+(z)$ arranged at the intersection of the $l$-th column and $s$-th row, respectively. $A_+(z)$~is a nilpotent matrix for each $z \in \T$ with an index $l(z)$ $( \le d)$.

Thus we have shown the following auxiliary result.

\begin{lemma} \label{2.2:lem} Let $\bA$ be the block  convolution operator defined by \eqref{2.1:op_splot} for which condition \eqref{2.1:war_zb} is fulfilled. Then its symbol $A(z)$ admits  a triangular decomposition
	\begin{equation*}
	A(z) = A_0(z) + A_+(z), \quad z \in \T,
\end{equation*}	
where $A_0(z)$ is the diagonal of $A(z)$ with respect to the chain $\pi(z)$	defined by \eqref{2.4:lan_skoncz}, and $A_+(z)$ is an upper triangular matrix representing an nilpotent ope\-ra\-tor on $\C^d$.
\end{lemma}

Due to the fact that the eigenvalues and vectors of the orthonormal Schur basis of the symbol $A(z)$ $(z \in\ T)$ were chosen to be continuous functions of~$z$, $T(z)$, $U(z)$, $(U(z))^* = (U(z))^{-1} $ ($z \in T$) represent also continuous matrix-valued functions. The corresponding convolution operators be denoted by  $T$, $\bU$ and $\bU^{-1}$, respectively. In accordance with \eqref{2.4:unit_rown_T}, we have
\begin{equation*}
	\bA = \bU \, T \, \bU^{-1}.
\end{equation*}
Clearly,
\begin{equation*}
	\bU = \sum_{n = -\infty}^{\infty} U_n \, \bS^n, \quad 	\bU^{-1} = \sum_{n = -\infty}^{\infty} (U^{-1})_n \, \bS^n, \quad T = \sum_{n = -\infty}^{\infty} T_n \, \bS^n
\end{equation*}
where
\begin{equation*}
	U_n = \int_\T U(z) z^{-n} d \tau, \quad (U^{-1})_n = \int_\T (U(z))^{-1} z^{-n} d \tau, \quad T_n = \int_\T T(z) z^{-n} d \tau,
\end{equation*}
\noindent $n \in \Z$. Likewise, $\Delta P_k(z)$ represents the corresponding  symbols for the operators 
\begin{equation*}\label{2.2:delta_P_k_splot}
	\Delta \bP_k = \sum_{n = -\infty}^{\infty} (\Delta P_k)_n \, \bS^n, \quad k = 1, \dots, d,
\end{equation*}
where
\begin{equation*}\label{2.2:delta_P_k_wspl}
	(\Delta P_k)_n = \int_\T \Delta P_k (z) z^{-n} d \tau, \quad n \in \Z, \quad k= 1, \dots, d.
\end{equation*}

The correspondence between the convolution operators and their symbols allows us to conclude that $\Delta \bP_k$ $(k = 1, \dots, d)$ are projections on the space $\X$ and, moreover,
\begin{equation*}\label{2.4:splot_ortog}
	\Delta \bP_k \, \Delta \bP_j = 0, \quad k \neq j, \quad k, j = 1, \dots, d,
\end{equation*}
\begin{equation*}\label{2.4:splot_zup}
	\sum_{k=1}^d \Delta \bP_k = \bI.
\end{equation*}
The projections $\Delta \bP_k $ $(k = 1, \dots, d)$ decompose the space $\X$ into a direct sum of subspaces
\begin{equation*}
	\X =  \X_1 \oplus \dots \oplus \X_d,
\end{equation*}
\noindent where
\begin{equation*}
	\X_k := \Delta \bP_k \X, \quad k = 1, \dots, d.
\end{equation*}
Denoting
\begin{equation*}
	\bP_k = \sum_{l=1}^k \Delta \bP_l, \quad k = 1, \dots, d,
\end{equation*}
we obtain the finite chain
\begin{equation*}
	\pi: \quad  O = \bP_0 \ < \ \bP_1 \ < \ \dots \ < \ \bP_k \ < \ \dots \ < \bP_d = \bI
\end{equation*}
on $\X$. The chain $\pi$ is invariant under the operator $\bA$, and the diagonal $\bA_0$ of $\bA$ with respect to $\pi$ is a convolution operator with the symbol $A_0(z)$, $z \in \T$. 
 
Next, we consider the family of curves $\Gamma_k$ obtained, respectively, as images of the functions $\lambda_k = \lambda_k(z)$ ($z \in \T; \, k = 1, \dots, d$), that is
\begin{equation*}
	\Gamma_k = \lambda_k(\T), \quad k = 1, \dots, d,
\end{equation*}
and consider the set
\begin{equation*}
	\Gamma = \bigcup_{k=1}^{d} \Gamma_k.
\end{equation*}
Notice that the closed set $\overline \Gamma$ coincides with the spectrum of the operator $\bA$. 

Having the numbering of the curves $\Gamma_k$ and choosing an appropriate direction of motion on each of them, we transform the set $\Gamma$ into an ordered set of points on the complex plane. Herewith, for any points $\nu$ and $\mu$ in $\Gamma$, relation $\nu \prec \mu$ will mean that $\nu$ belongs to a curve with a lower number than $\mu$, or $\nu$ and $\mu$ lie on the same curve, but $\nu$  precedes $\mu$ in the course of the selected (positive) direction. Representing a point $\nu \in \Gamma_k$ as $\nu = \lambda_k(e^{i t_{\nu}})$, where $0 \le t_{\nu}  < 2 \pi$, we shall take the ordered for the points $\nu$ on $\Gamma_k$ being corresponding to the increment of the argument $t_\nu$ on the interval $\left[0, 2 \pi\right).$

In accordance with above arguments, we shall write $\nu \prec \mu$ if $\nu \in \Gamma_k$ and $\mu \in \Gamma_l$ with $k < l$ or $\nu, \, \mu \in \Gamma_k$ for which $0 \le t_{\nu}  < t_{\mu} < 2 \pi$. We let $\alpha_k $ $\left(= \lambda_k(e^{i0})\right)$ for the beginning point of the curve $\Gamma_k$ and $\beta_k $ $\left(= \lambda_k(e^{2 \pi i})\right).$

Next, we give a construction of a maximal chain on the space $\X$ that is invariant under the block convolution operator $\bA$ given by \eqref{2.1:op_splot}. Let $\nu \in \Gamma$, \linebreak i.e.,  $\nu \in \Gamma_k$ for a $k \in \{1, \dots, d\}$, and consider the matrix-valued function
\begin{equation}\label{2.4:lancuch_P_nu}
	P_{\nu} (z) := P_{k-1}(z) + \chi_{[\alpha_k, \nu)}(\lambda_k(z)) \Delta P_k(z),
\end{equation}
\noindent where $P_{k}(z)$, $\Delta P_{k}(z)$ ($k = 1, \dots, d$) are defined by \eqref{2.4:proj}, \eqref{2.4:skok}, respectively. Note that
\begin{equation*}
	P_{\alpha_1}(z) = O,
\end{equation*}
and
\begin{align*}
P_{\beta_d} (z) & = P_{d-1}(z) + \chi_{[\alpha_d, \beta_d)}(\lambda_d(z)) \Delta P_d(z) \\
& = \sum_{l=1}^{d-1}\Delta P_l(z) + \Delta P_d(z) = \sum_{l=1}^{d}\Delta P_l(z) = I,
\end{align*}
i.e., 
\begin{equation*}
	P_{\beta_d}(z) = I.
\end{equation*}
$P_{\nu} (z)$ are orthogonal projections on $\C^d$ and, as can be easily seen,
\begin{equation*}\label{2.4:niezm_APnu}
	P_{\nu} (z)A(z)P_{\nu} (z) = A(z)P_{\nu} (z).
\end{equation*}

\noindent Moreover, by using \eqref{2.4:rozklad_mac_A0} and \eqref{2.4:lancuch_P_nu}, we get
\begin{equation*}
	A_0(z)P_{\nu} (z) = \sum_{l=1}^{k-1} \lambda_l(z) \Delta P_l(z) + \chi_{[\alpha_k, \nu)}(\lambda_k(z)) \lambda_k(z)\Delta P_k(z), \quad z \in \T,
\end{equation*}
\noindent and moreover
\begin{equation*}
	A_0(z)P_{\nu} (z) = P_{\nu} (z) A_0(z), \quad z \in \T.
\end{equation*}

\noindent As a result, we obtain the following chain
\begin{equation*}
	O = P_{\alpha_1}(z) \ < \ P_{\nu}(z) \ < \  P_{\mu}(z) \ < \ P_{\beta_d}(z) = I, \quad \nu \prec \mu, \ \nu, \, \mu \in \Gamma, \ z \in \T,
\end{equation*}
which, for each $z \in \T$, is invariant under $A(z)$ and is reducing for $A_0(z)$.

Each of the matrix-valued functions $P_{\nu} (z),$ $z \in \T,$
is the symbol of the convolution operator 
\begin{equation*}
	\bP_{\nu}  = \sum_{n \in \Z} P_{\nu, n} \, \bS^n, \quad \nu \in \Gamma,
\end{equation*}
where $P_{\nu, n}$ are the Fourier coefficients
\begin{equation*}
	P_{\nu, n} = \int_{\T} P_{\nu}(z) z^{-n} d \tau, \quad n \in \Z.
\end{equation*}

From the properties of the projections $P_{\nu} (z)$ it immediately follows that that $\bP_{\nu} $ represents projections on the space $\X$, i.e.,
\begin{equation*}
	\bP_{\nu}^2  = \bP_{\nu}, \quad \nu \in \Gamma,
\end{equation*}
and, in addition,
\begin{equation*}
	\bP_{\nu} \, \bP_{\mu} = \bP_{\nu}, \quad \nu \prec \mu, \quad \nu, \mu \in \Gamma,
\end{equation*}
and
\begin{equation*}
	\bP_{\alpha_1} = O, \quad \bP_{\beta_d} =  \bI,
\end{equation*}
as well.

Therefore, we obtain a chain
\begin{equation*} 
	\bbP: \quad  O = \bP_{\alpha_1} \ < \  \bP_{\nu} \ < \ \bP_{\mu} \ < \ \bP_{\beta_d}  = \bI, \quad \nu \prec \mu, \quad   \nu, \mu \in \Gamma,
\end{equation*}
defined on the space $\X$. The chain $\bbP$, thus obtained, is maximal, and invariant under the block convolution operator $\bA$ given by formula \eqref{2.1:op_splot}.

Based on Lemma \ref{2.2:lem} we conclude that the operator $\bA$ admits the decomposition 
\begin{equation}\label{2.2:dekompoz_op_splot}
	\bA = \bA_0 + \bA_+,
\end{equation}
where $\bA_0$ and $\bA_+$ represent convolution operators corresponding to the symbols $A_0(z)$ and $A_+(z)$, respectively. The operators $\bA_0$ and $\bA_+$ are respectively determined by 
\begin{equation*}
	\bA_0  = \sum_{n \in \Z} A_{0, n} \, \bS^n,
\end{equation*}
where
\begin{equation*}
	A_{0, n} = \int_{\T} A_0(z) z^{-n} d \tau, \quad n \in \Z,
\end{equation*}
and
\begin{equation*}
	\bA_+  = \sum_{n \in \Z} A_{+, n} \, \bS^n,
\end{equation*}
where
\begin{equation*}
	A_{+, n} = \int_{\T} A_+(z) z^{-n} d \tau, \quad n \in \Z. \\
\end{equation*}

The chain $\bbP$ is invariant under the operator $A_+$, and is reducing for $A_0$. The representation \eqref{2.2:dekompoz_op_splot} is nothing than the upper triangular decomposition of $A$ with respect to the chain $\bbP$. $A_0$ is the corresponding diagonal operator of $A$ and it can be expressed by the integral (cf. \eqref{1.2:repr_diag_calk_lan})
\begin{equation}\label{2.2:repr_calk_A_0}
	\bA_0 =  \int_{\bbP} (dP) \bA (dP),
\end{equation}
along the chain $\bbP$. Accordingly (cf. \eqref{1.2:repr_LU_calk_lan})
\begin{equation}\label{2.2:repr_calk_A_+}
	\bA_+ =  \int_{\bbP} P \bA (dP).
\end{equation}
The validity of the integral representations \eqref{2.2:repr_calk_A_0} and \eqref{2.2:repr_calk_A_+} are established by applying direct arguments.

Since $A_+(z)$ is a nilpotent matrix with index $l(z)$ for each $z \in \T$, it follows that the operator $\bA_+$ is respectively a nilpotent operator of index
\begin{equation}\label{2.2:index}
l = \max_{z \in \T} l(z).
\end{equation}

\begin{theorem} \label{3:tw_3.3} Let $\bA$ be the block convolution operator on $\X$ given by \eqref{2.1:op_splot} whose elements satisfy \eqref{2.1:war_zb}. Then the operator $\bA$ admits the triangular decomposition \eqref{2.2:dekompoz_op_splot}, i.e.,
	\begin{equation*}
		\bA = \bA_0 + \bA_+,
	\end{equation*}
	where $\bA_0$ is the diagonal of  $\bA$ with respect to the chain $\bbP$, and $A_+$ is a triangular with respect to $\bbP$ nilpotent operator of index \eqref{2.2:index}. The operator $\bA_0$ is expressed by the integral \eqref{2.2:repr_calk_A_0} and $\bA_+$ by \eqref{2.2:repr_calk_A_+}.
\end{theorem}

Although the additive triangular representation \eqref{2.2:dekompoz_op_splot} seems rather general, useful information regarding spectral properties of the operator $\bA$ can be described directly.

\begin{theorem} The following statements hold true.
	\begin{enumerate}[label={\textup{(\roman*)}}, widest=iii, leftmargin=*]
	\item 	$\sigma(\bA) = \sigma(\bA_0) = \bigcup_{j=1}^d \overline{\lambda_j(\T)}.$	
	
	\item Spectrum of the restriction $\bA \bP_\nu$ of the operator $\bA$ to the subspace $\bP_\nu \X$ consists of the closure of the set of all points $\mu \in \Gamma$ proceeding $\nu$, and the spectrum of the operator $(I-\bP_\nu)\bA$ on $(I-\bP_\nu) \X$ consists of the closure of all $\mu \in \Gamma$ following $\nu$.
\end{enumerate}
	
\end{theorem}

\begin{proof}
	(i) Note that $\lambda$ is a regular point of the operator $\bA$ if and only if
	\begin{equation*}
		\det(A(z) - \lambda I) \neq 0, \quad \textrm{for all} \ z \in \T.
	\end{equation*}
Since	
		\begin{equation*}
		\det(A(z) - \lambda I) = 	\det(A_0(z) - \lambda I) = \prod_{j=1}^d (\lambda_j(z) - \lambda),
	\end{equation*}
where $\lambda_j(z)$ denote the eigenvalues of $A(z)$ for $z \in \T$, the assertion follows.

(ii) Similar arguments can be applied to the operators $\bA \bP_\nu$, $(I - \bP_\nu)\bA $ which in turn represent convolution operators corresponding to the symbols $\bA(z) \bP_\nu(z)$, $(I - \bP_\nu(z))\bA(z)$, respectively.
\end{proof}

The statement in (ii) expresses the separation property of the spectrum with respect to the chain $\bbP$ for the operator $\bA$.

\section{Hilbert space case} \label{2.2:Hilbert}

In this section, the described spectral representation will be adapted and spe\-ci\-fied for the special case of Hilbert space.

The block convolution operator $\bA$  defined by \eqref{2.1:op_splot} under the condition \eqref{2.1:war_zb} is assumed to be acting on the Hilbert space of (block) sequences $\eld$. The symbol $A(z),$ $z \in \T$, of the operator $\bA$ is a matrix-valued function that belongs to the Hilbert space $\ELDM$ whose Fourier coefficients are $A_n$, $n \in \Z$, that is, those given by \eqref{2.1:wspl_Fourier_znorm}. The relationship between operator $\bA$ and its symbol $A(z),$ $z \in \T$, is given by the Fourier transformation $\bF$ on $\EL$, namely,
\begin{equation*}
	(\bF^* \bA \bF u)(z) = \bA(z)u(z), \quad z \in \T,
\end{equation*}
where $u \in \EL$. Note that $\bF u$ represents the sequence $(u_n)$ of elements $u_n$, where $u_n$ are the Fourier coefficients of the expansion of the function \linebreak $u \in \EL$ in the Fourier series, i.e., whenever
\begin{equation*}
	u(z) = \sum_{n \in \Z}u_n z^n
\end{equation*}
with
\begin{equation*}
	u_n = \int_{\T} u(z) z^{-n} dz, \quad n \in \Z.
\end{equation*}

In the framework of taken approaches, following arguments applied in the previous section, on the space $\EL$ one can define the multiplication ope\-rators
\begin{align}\label{2.2:op_mnoz_symb}
    (\bfA u)(z) &= A(z)u(z),\\
	\label{2.2:op_mnoz_P_k}({\bf P}_ku)(z) & = P_k(z) u(z),  \\
	\nonumber (\Delta {\bf P}_ku)(z) & = \Delta P_k(z) u(z), 
\end{align}
for $z \in \T$ a.e.; $k =1, \dots, d,$ and $u \in \EL$, where $A(z)$, $\Delta P_k (z)$, $P_k (z)$ are given by \eqref{2.1:def_mac}, \eqref{2.4:skok}, \eqref{2.4:proj}, respectively. The operators $\bfP_k$, $\Delta {\bf P}_k$ $(k=1, \dots, d)$ are orthogonal projections on $\EL$. Moreover, owing to \eqref{2.4:skok}, \eqref{2.4:proj} and \eqref{2.4:monot},  we get the property of monotonicity for the family of the projections
\begin{equation*}
	\bfP_k = \sum_{l=1}^k \Delta \bfP_l, \quad k=1, \dots, d,
\end{equation*}
i.e.,
\begin{equation*}
\bfP_k \bfP_l = \bfP_l \bfP_k =\bfP_k, \quad k < l; \ k,l = 1, \dots, d,
\end{equation*}
hence
\begin{equation}\label{4.1:lancuch_P_k}
	\overline \bfP: \quad  \bfO = \bfP_0 \ < \ \bfP_1 \ < \ \dots \ < \ \bfP_d = \bfI
\end{equation}
is a chain of orthogonal projections on $\EL $.

Taking into account \eqref{2.2:op_mnoz_symb} and \eqref{2.2:op_mnoz_P_k}, due to relation \eqref{2.4:niezm_mac}, we have
\begin{equation*}
	({\bfP}_k\bfA \bfP_k u)(z) = P_k(z) A(z) P_k(z) u(z) = A(z)P_k(z) u(z) = (\bfA \bfP_k u)(z),
\end{equation*}
for any $z \in \T$ and any $k = 1, \dots, d$. So
\begin{equation}\label{4.1:niezm_lanc}
	{\bfP}_k\bfA \bfP_k = \bfA \bfP_k, \quad k = 1, \dots, d,
\end{equation}
which means that the chain defined by \eqref{4.1:lancuch_P_k} is invariant under the operator $\bfA$.

Further, we define the operators $\bfA_0$ and $\bfA_+$ on the space $\EL$ by 
\begin{equation*}
	(\bfA_0 u)(z) = A_0(z)u(z), \quad z \in \T,
\end{equation*}
and
\begin{equation*}
	(\bfA_+ u)(z) = A_+(z)u(z), \quad z \in \T,
\end{equation*}
where $A_0(z)$ and $A_+(z)$ are determined by \eqref{2.4:rozklad_mac_A0} and \eqref{2.4:rozklad_mac_A+}, respectively. Due to \eqref{2.4:rozklad_mac}, we obtain the additive decomposition
\begin{equation}\label{4.1:rozkl_op_mnoz_A}
	\bfA = \bfA_0 + \bfA_+.
\end{equation}

From the commutative relations of the matrix-valued function $A_0(z)$ and the invariance of $A_+(z)$ with respect to the family of projections  \linebreak $P_k(z)$ $(k = 1, \dots, d)$ it follows that the chain \eqref{4.1:lancuch_P_k} is invariant under the operator $\bfA_+$ and, accordingly, is reducing for $\bfA_0$, i.e.,
\begin{equation}\label{4.1:niezm_lanc_A_+}
	{\bfP}_k\bfA_+ \bfP_k = \bfA_+ \bfP_k, 
\end{equation}
and
\begin{equation}\label{4.1:red_lanc_A_0}
\bfP_k \bfA_0  = \bfA_0 \bfP_k,
\end{equation}
for $k = 1, \dots, d$.

The projections $\Delta {\bf P}_k \, (k=1, \dots, d)$ are mutually disjoint, i.e.,
\begin{equation*}
\Delta {\bf P}_k \Delta {\bf P}_l = \delta_{kl}\bfI, \quad   k, l=1, \dots, d,
\end{equation*}
and form a complete system
\begin{equation*}
	\sum\limits_{k=1}^d \Delta {\bf P}_k = \bfI.
\end{equation*}
Therefore, the space $\cH = \EL$ admits the orthogonal decomposition
\begin{equation}\label{4.1:decomp_H}
	\cH = \sum\limits_{k=1}^d \bigoplus \cH_k,
\end{equation}
where
\begin{equation*}
	\cH_k := \Delta {\bf P}_k \cH, \quad k=1, \dots, d.
\end{equation*}

\noindent From \eqref{4.1:niezm_lanc_A_+} and \eqref{4.1:red_lanc_A_0} it follows that the decomposition \eqref{4.1:decomp_H} of $\cH$ consists of invariant subspaces for $\bfA$, and reducing for $\bfA_0$. Moreover, 
\begin{equation*}
	\bfA_0 \Delta {\bf P}_k = \lambda_k(z) \Delta {\bf P}_k, \quad k = 1, \dots, d.
\end{equation*}
Denoting $\bfA_{0, k} = \bfA_0|_{\cH_k}$, where $\bfA_0|_{\cH_k}$ designate the part of $\bfA_0$ on $\cH_k$, one can write
\begin{equation}\label{2.4:A0_suma_ort}
	\bfA_0 = \sum\limits_{k=1}^d \bigoplus \bfA_{0, k}.
\end{equation}

Next, consider the operators $V_k: \cH_k \longrightarrow \ELj$ defined by 
\begin{equation*}
(V_ku)(z) := \left\<u(z), \varphi^{(k)}(z)\right\>, \quad z \in \T\  a.e.,
\end{equation*}
for $u \in \cH_k$ ($k=1, \dots, d$). It is seen that each of the operators $V_k$ is invertible, and
\begin{equation*}
	(V_k^{-1}\varphi)(z) := \varphi(z)\varphi^{(k)}(z), \quad z \in \T \ a.e.,
\end{equation*}
for $\varphi \in \ELj$ and $k = 1, \dots, d$. Since
\begin{equation*}
	V_k\bfA_{0, k} = \bfA_{k}V_k, \quad k=1, \dots, d,
\end{equation*}
where $\bfA_k$ denotes the multiplication operator by $\lambda_k(z)$ in the space $\ELj$, i.e.,
\begin{equation}\label{4.1:H_mult}
	(\bfA_{k}\varphi)(z)= \lambda_k(z) \varphi(z), \quad z \in \T \ a.e.,
\end{equation}
for $\varphi \in \ELj$, due to \eqref{2.4:A0_suma_ort}, the following decomposition for the operator $\bfA_0$ holds
\begin{equation}\label{4.1:rozkl_A_0}
	\bfA_0 = \sum\limits_{k=1}^d \bigoplus V_k^{-1} \bfA_kV_k.
\end{equation}

Now, for every $k \in 1, \dots, d$, we give a construction of a chain defined  on  the space $\ELj$ reducing the operator  $\bfA_{k}$. To this end, recall that each contour \linebreak $\Gamma_k = \lambda_k(\T)$ is oriented according to the increase of the argument $z$ on $\T$. \linebreak Accordingly, for some point $\nu \in \Gamma_k$ we define on the space $\ELj$ the operator
\begin{equation*}
	(\bfP_{k, \nu} \varphi)(z) = \chi_{[\alpha_k, \nu)}(\lambda_k(z))\varphi(z), \quad z \in \T \ a.e.,
\end{equation*}
for $\varphi \in \ELj$. It is clear that $\bfP_{k, \nu} $ represents an orthogonal projection on the space $\ELj$. In addition, for any $\nu, \mu \in \Gamma_k$ such that $\nu \prec \mu$, the relation $\bfP_{k, \nu} < \bfP_{k, \mu}$ holds true. Thus, we obtain the following chain
\begin{equation}\label{4.1:chain_P_k,nu_H}
\bfO = \bfP_{k, \alpha_k} \ <  \ \bfP_{k, \nu}  \ < \ \bfP_{k, \mu} \ <  \ \bfP_{k, \beta_k} = \bfI, \quad \nu \prec \mu, \ \nu, \mu \in \Gamma_k.
\end{equation}
Since, in view of \eqref{4.1:H_mult} and \eqref{4.1:rozkl_A_0} , 
\begin{align*}
	(\bfP_{k, \nu} \bfA_{k} \varphi)(z) = \lambda_k(z)\chi_{[\alpha_k, \nu)}(\lambda_k(z))\varphi(z) = ( \bfA_{k}\bfP_{k, \nu} \varphi)(z)
\end{align*}
for every $\varphi \in \ELj$,  $\nu \in \Gamma_k$ and $z \in \T \ a.e.$, it follows
\begin{equation*}
	\bfP_{k, \nu}\bfA_{k} =  \bfA_{k}\bfP_{k, \nu}, \quad  \nu \in \Gamma_k,
\end{equation*}
hence, the chain defined by \eqref{4.1:chain_P_k,nu_H} is reducing for $\bfA_k$ for each $k = 1, \dots, d.$

Next, we let
\begin{equation*}
	\bfP_{\nu} = \sum\limits_{k=1}^d V_k^{-1} \bfP_{k, \nu} V_k, \quad \nu \in \Gamma.
\end{equation*}
$\bfP_{\nu}$ are orthogonal projections on $\cH$. For any $\nu \in \Gamma_k$ the projection $\bfP_{\nu}$ can be written in the form
\begin{equation} \label{4.1:H_p_nu}
	\bfP_{\nu} =  {\bf P}_{k-1} + \chi_{[\alpha_k, \nu)}(\lambda_k(\cdot))\Delta {\bf P}_k.
\end{equation}
It is clear that $\bfP_{\nu} < \bfP_{\mu}$ whenever $\nu \prec \mu$, hence
\begin{equation}\label{4.1:H_lan_symb}
	\bfP: \quad  \bfO = \bfP_{\alpha_1 } \ < \ \bfP_\nu \ < \ \bfP_\mu \ <  \bfP_{\beta_d}  = \bfI, \quad \nu \prec \mu\quad \nu, \mu \in \Gamma,
\end{equation}
forms a chain on the space $\cH$ ($= \EL$).

Since
\begin{equation*}
	\bfP_\nu \bfA_0  = \bfA_0 \bfP_\nu,\quad \nu \in \Gamma
\end{equation*}
the chain $\bfP$ given by \eqref{4.1:H_lan_symb} is reducing for the operator  $\bfA_0$. Each of the subspaces $\bfP_k \cH$ $(k = 1, \dots, d)$ is invariant with respect to the operator $\bfA$ (cf. \eqref{4.1:niezm_lanc}). According to representation \eqref{4.1:H_p_nu}, we conclude that the chain $\bfP$  is invariant under the operator $\bfA$.

As a result, we obtain that the multiplication operator $\bfA$ defined on the space $\EL$ admits the additive upper triangular decomposition \eqref{4.1:rozkl_op_mnoz_A} relative to the chain $\bfP$ given by \eqref{4.1:H_lan_symb}.

If $\nu \in \Gamma$, that is, $\nu \in \Gamma_k$ for some $k = 1, \dots, d$, then, as is easily seen, the restriction of the operator $\bfA_0$ to the subspace $\bfP_\nu \cH$ (recall that $\cH := \EL$) is of the form
\begin{equation}\label{4.1:H_post_A0P_nu}
\bfA_0 \bfP_{\nu} =  \sum\limits_{l=1}^{k-1} \lambda_l(\cdot)\Delta {\bf P}_l + \chi_{[\alpha_k, \nu)}(\lambda_k(\cdot)) \lambda_k(\cdot) \Delta {\bf P}_k.
\end{equation}
From the relation \eqref{4.1:H_post_A0P_nu}, taking into account that $\bfA_0$ is the diagonal of the operator $\bfA$ with respect to the chain $\bfP$, we conclude that
\begin{equation*}
	\sigma(\bfA \bfP_{\nu})	=  \overline{\big\{\mu \in  \Gamma: \mu \prec \nu \big\}},
\end{equation*}
and, similarly,
\begin{equation*}
	\sigma(\bfA (\bfI - \bfP_{\nu}) ) = \overline{\big\{\mu \in  \Gamma: \nu \prec \mu \}},
\end{equation*}
which reflects the property of separation of the spectrum of the operator $\bfA$ with respect to the chain $\bfP$.

Next, we let
\begin{equation}\label{4.1:H_splot_P_nu}
	\bP_\nu = \bF \, \bfP_\nu \, \bF^*, \quad \nu \in \Gamma,
\end{equation}
where $\bF$ denotes the Fourier operator acting from $\EL$ onto $\eld$, which to the element $u$ in $\EL$ assigns the sequence $(u_n)$ of its corresponding Fourier coefficients. $\bP_\nu$ are orthogonal projections on the space $\eld$ for which
\begin{equation*}
	\bP_\nu \bA \bP_\nu = \bA \bP_\nu, \quad \nu \in \Gamma.
\end{equation*}
Moreover,
\begin{equation*}
	\bbP: \quad  O = \bP_{\alpha_1} \ < \  \bP_\nu  \ < \  \bP_\mu \ < \  \bP_{\beta_d} = \bI, \quad \nu \prec \mu, \quad \nu, \mu \in \Gamma,
\end{equation*}
is a chain invariant under the block convolution operator $\bA$ defined by \eqref{2.1:op_splot}. \linebreak In turn, the operator $\bA$ admits an additive upper triangular decomposition
\begin{equation}\label{4.1:H_decomp_A}
	\bA = \bA_0 + \bA_+,
\end{equation}
where $	\bA_0 = \bF \, \bfA_0 \, \bF^*$ and $	\bA_+ = \bF \, \bfA_+ \, \bF^*$ is the upper triangular part of $\bA$ relative to the chain $\bbP$. Note that $\bA_+$ is a nilpotent operator of an index $l \le d$. Clearly, the separation property for the spectrum of the operator $\bA$ is also fulfilled.

We have proved the following result.

\begin{theorem} \label{2.2:tw_hilbert} Let $\bA$ be the block convolution operator defined on the Hilbert space $\eld$ by \eqref{2.1:op_splot}. Then
	
	\begin{enumerate}[label={\textup{(\roman*)}}, widest=iii, leftmargin=*]
		\item the chain $\bbP$ of orthogonal projections $\bP_\nu$ $(\nu \in \Gamma)$ given by \eqref{4.1:H_splot_P_nu} is maximal on $\eld$, and invariant under the operator $\bA$;
		\item  $\bA$ admits the additive upper triangular decomposition \eqref{4.1:H_decomp_A} with respect to $\bbP$, where
		\begin{equation*}
			\bA_0 =  \int_{\bbP} (d P) \bA (d P),
		\end{equation*}
	is the diagonal of $\bA$, and the corresponding upper triangular part $\bA_+$ of $\bA$ is a nilpotent operator of an index $l \le d$;	
		\item the separation property for the spectrum of  $\bA$ holds true. Moreover
		\begin{equation*}
			\sigma(\bA) = 	\sigma(\bA_0) = \bigcup_{k=1}^d \overline {\lambda_k(\T)},
		\end{equation*}
	where $\lambda_k(z)$ denote the eigenvalues of the symbol $A(z)$, $z \in \T.$
	\end{enumerate}
 
\end{theorem}

\section{Applications. Jacobi matrices}

In this section we shall deal with the operators generated by the so-called periodic Jacobi matrices. This kind of matrices appear naturally in the many problems of mathematics, physics, mathematical physics, etc. As a particular case, but of a special merit, it can be mentioned the Born-von Karman model of the crystal lattice (cf. \cite{Glaz}). In the simplest case, this model is nothing more than a one-dimensional infinite chain of atoms. The equations of motion in the Born-von Karman model do not differ from those for beads on an elastic string and have the form
\begin{equation*}
	\frac{d^2 u}{dt^2} = \tilde A u,
\end{equation*}
where $\tilde A = [a_{rs}]_{r,s \in \mathbb{Z}}$ is a complex a \textit{$d$-periodic banded Jacobi matrix of order $k$}, i.e., the matrix whose entries satisfy the following conditions:
\begin{enumerate}[label={\textup{(\roman*)}}, widest=iii, leftmargin=*]
	\item there exists $k \in \N$ such that for all $r,s \in \mathbb{Z}$ we have $a_{rs}  = 0$ if $|r-s| > k$,
	\item there exists $r, s \in \mathbb{Z}$ such that $a_{rs}  \neq 0$ if $|r-s| = k$,
	\item  there exists $d \in \N$ such that for all $r, s \in \mathbb{Z}$ we have  $a_{r + d s + d} = a_{rs}$.
\end{enumerate}
Finding the eigenfrequencies is reduced to the spectral analysis of the periodic Jacobi matrix $\tilde A$.

Denote by $A_l$ ($l \in \Z$) a matrix with entries $(A_l)_{rs} = a_{rs}$ \linebreak ($r = 0, \ldots, d-~1$ and $s = -d l, \ldots, d(1-l)-1$). Note that $A_l \in M_d(\C)$ $(l \in \Z)$ and  $A_l = 0$ for $|l| > m$, where $m = \lceil \frac{k}{d} \rceil$ (i.e. $m$ is the smallest integer not less than $\frac{k}{d}$). Under the above notation, the matrix $\tilde A$ has a block form 
\begin{equation*}
	A =\left[ \begin{array}{ccccccccccc}
		\ddots &  & \ddots & \ddots & \ddots &  & \ddots &  &  &  &    \\
		& A_m & \dots & A_1 & A_0 & A_{-1} & \dots &  A_{-m} &  &  & \\
		&  & A_m & \dots  & A_1 & A_0 & A_{-1} & \dots &  A_{-m} &  &   \\
		&  &  & A_m  & \dots & A_1 & A_0 & A_{-1} & \dots &  A_{-m} &   \\
		&  &  &  & \ddots &  & \ddots & \ddots & \ddots &  & \ddots   \\
	\end{array} \right ].
\end{equation*}

Consider the operators 

\begin{equation} \label{3.2:Laurent-definicja-dwa}
	(\tilde A v)_n = \sum_{j = n-k}^{n+k} a_{nj}v_{j}, \quad  n \in \Z,
\end{equation}
where $v = (v_n) \in \elpj$, and
\begin{equation} \label{3.2:Laurent-definicja-trzy}
	(\bA u)_n = \sum_{j=n-m}^{n+m} A_{n-j} u _{j}, \quad n \in \Z,
\end{equation}
for $u = (u_n) \in \elp$, acting on the Banach spaces $\elpj$ and $\elp$ ($1 \le p < \infty$)  and corresponding to the matrices $\tilde A$ and $A$, respectively. Using the matrix reduction procedure from \cite[pp. 139--140]{cojuhari} we can show that the operator $\tilde A$ defined by \eqref{3.2:Laurent-definicja-dwa} is similar (and for $p = 2$ -- unitarily equivalent) to the convolution operator $\bA$ defined by \eqref{3.2:Laurent-definicja-trzy}. Since 
\begin{equation*} 
	\sum_{j=-m}^{m} |A_{j}| < \infty,
\end{equation*}
it follows that $A$ is bounded, and its symbol $A(z),$ $z \in \T$,  is a matrix-valued rational function determined by
\begin{equation*}
	A(z) = \sum_{j=-m}^{m} A_{j} z^j, \quad z \in \T.
\end{equation*}
Since the operator $A$ defined by \eqref{3.2:Laurent-definicja-trzy} satisfies the assumptions of Theorem \ref{3:tw_3.3}, then it admits the triangular decomposition \eqref{2.2:dekompoz_op_splot}.

Note that, using the matrix reduction procedure described above to the ope\-rator $A$ given by \eqref{3.2:Laurent-definicja-trzy}, we get a convolution operator acting on the space $l_p(\Z; \C^{md})$ $(1 \le p < \infty)$ corresponding to a tridiagonal block matrix, which is similar (and for $p = 2$ -- unitarily equivalent) to $\bA$. Obviously, if we assume that
\begin{equation*}
J_1 = \left[ \begin{array}{cccc}
		A_m  & A_{m-1}  &   \cdots   &    A_{1}   \\
		0  & A_{m}  &   \cdots   &    A_{2} \\
		\vdots   & \vdots  & \ddots  &   \cdots        \\
		0  & 0  &   \cdots   &    A_{ m} \\
	\end{array} \right ], \quad
J_{-1} = \left[ \begin{array}{cccc}
		A_{-m}  & 0  &   \cdots   &    0   \\
		A_{-m + 1}  & A_{-m}  &   \cdots   &    0 \\
		\vdots   & \vdots  & \ddots  &   \cdots        \\
		A_{ - 1}  & A_{- 2}  &   \cdots   &    A_{ -m} \\
	\end{array} \right ], 
\end{equation*}

\begin{equation*}
	J_0 = \left[ \begin{array}{cccc}
		A_0  & A_{-1}  &   \cdots   &    A_{ -m + 1}   \\
		A_1  & A_{0}  &   \cdots   &    A_{ -m + 2} \\
		\vdots   & \vdots  & \ddots  &   \cdots        \\
		A_{m - 1}  & A_{m - 2}  &   \cdots   &    A_{ 0} \\
	\end{array} \right ], 
\end{equation*}
then $J_{-1}, J_0, J_1 \in M_{md}(\C),$ and the matrix $A$ has the following block tridiagonal form
\begin{equation}\label{3.2:Laurent-definicja-piec}
	J = \left[ \begin{array}{cccccc}
		\ddots & \ddots &  &  &  &    \\
		\ddots & J_0 & J_{-1} &  &  &    \\
		& J_1 & J_0 & J_{-1} &  &    \\
		& & J_1 & J_0 & \ddots &    \\
		&  &  & \ddots & \ddots &    \\
	\end{array} \right ].
\end{equation}
It follows that an operator 
\begin{equation} \label{3.2:Laurent-definicja-cztery}
	(J u)_n = \sum_{j=n-1}^{n+1} J_{n-j} u _{j}, \quad n \in \Z
\end{equation}
where $u = (u_n) \in l_p(\Z; \C^{md})$, is similar to the operator $A$ given by \eqref{3.2:Laurent-definicja-trzy}. Therefore, the operator $\tilde A$ defined by \eqref{3.2:Laurent-definicja-dwa} acting on the Banach space $\elj$, corresponding to the $d$-periodic banded Jacobi matrix of order $k$, is similar to the block convolution operator \eqref{3.2:Laurent-definicja-cztery} acting on the Banach space $l_p(\Z; \C^{md})$, corresponding to the tridiagonal block matrix \eqref{3.2:Laurent-definicja-piec}.

Let $\tilde A = [a_{r s}]_{r,s \in \mathbb{Z}}$ be a $d$-periodic banded Jacobi complex matrix of order~$1$ expressed as follows
\begin{equation}\label{3:najman_A}
	\tilde A = \left[ \begin{array}{ccccccccc}
		& \ddots      & \ddots &          &          &          &          &       &\\
		\ddots & b_d         & a_d    & c_d      &          &          &          &       &\\
		&    & b_1    & a_1      & c_1      &          &          &       &\\
		&     &        & \ddots   & \ddots   & \ddots   &          &       &\\
		&     &        &          & b_{d}  & a_{d}  & c_{d}  &       &\\
		&     &        &          &          & b_1      & a_1      & \ddots &\\
		&     &        &          &          &          & \ddots   & \ddots  &  \\
	\end{array} \right ].
\end{equation}

\noindent If we use the matrix reduction procedure for $\tilde A$, then we get its tridiagonal form
\begin{equation}\label{3:Najman_blok}
	\bA = \left[ \begin{array}{cccccc}
		\ddots & \ddots &  &  &  &    \\
		\ddots & A_0 & A_{-1} &  &  &    \\
		& A_1 & A_0 & A_{-1} &  &    \\
		& & A_1 & A_0 & \ddots &    \\
		&  &  & \ddots & \ddots &    \\
	\end{array} \right ],
\end{equation}
where $A_{1}, A_{-1}, A_0 \in M_d(\C)$ are expressed as
\begin{equation*}
	A_1 = \left[ \begin{array}{cccc}
		0 & \cdots &   0   &    b_1   \\
		0 & \cdots  & 0   &    0 \\
		\vdots  &   & \vdots   &   \vdots        \\
		0 & \cdots & 0 & 0
	\end{array} \right ], \quad
	A_{-1} = \left[ \begin{array}{cccc}
		0 &    0   &  \cdots &   0  \\
		\vdots  &    \vdots &  &   \vdots        \\
		0 & 0  & \cdots &  0       \\
		c_d  & 0 & \cdots &  0
	\end{array} \right ],
\end{equation*}
and
\begin{equation*}
	A_0 = \left[ \begin{array}{cccccc}
		a_1  & c_1  &   0   &    \cdots  &  0  \\
		b_2  & a_2  & c_2   &    \cdots  &  0  \\
		0   & b_3  & a_3   &   \cdots  &  0      \\
		\vdots     &  &  \ddots    & \ddots & \vdots\\
		0  &  \cdots  &       & b_{d }  & a_{d } \\
	\end{array} \right ],
\end{equation*}
respectively.

Repeating arguments presented above we can deduce that the operator defined by \eqref{3.2:Laurent-definicja-dwa} corresponding to the matrix $\tilde A$ given by \eqref{3:najman_A} is similar to the operator defined by \eqref{3.2:Laurent-definicja-trzy} corresponding to the matrix $A$ given by \eqref{3:Najman_blok}, so 
\begin{equation*}
	\sigma(\tilde A) = \sigma(A),
\end{equation*}
and $\tilde A$ admits the triangular decomposition \eqref{2.2:dekompoz_op_splot}.

Now, on the Banach space $\elp$ consider the operator $\bA$ of the form \eqref{3.2:Laurent-definicja-trzy} corresponding to the matrix $A$ given by \eqref{3:Najman_blok}. Its symbol is a matrix-valued rational function
\begin{equation*} 
	{A}(z) = A_{-1} z^{-1} + A_0 + A_1z, \quad z \in \T.
\end{equation*}
Since spectral properties of $\bA$ are related to properties of $A(z),$ $z \in \T$, we show how the eigenvalues of the matrix $A(z)$ depend on the selection of elements \linebreak of the matrices $A_{-1}$, $A_0$ and $A_1$.

Let $z \in \T$. The determinant of a matrix ${A}(z) - \lambda I$ has the following form
\begin{equation*}\label{3.2:wyzn_okr}
	\det \left(A(z) - \lambda I\right) = (-1)^{d}\big(Q(\lambda) - bz - cz^{-1}\big),
\end{equation*}
\noindent where $Q(\lambda)$ is a polynomial of degree $d$ of the variable $\lambda$, expressed as
\begin{align*}
	  Q(\lambda) & = (-1)^{d} \cdot \det \left[ \begin{array}{ccccc}
		a_1-\lambda  & c_1  &       \cdots  &  0  \\
		b_2  & a_2 -\lambda  &     \cdots  &  0  \\
		\vdots     &  &  \ddots    & \vdots\\
		0  &  \cdots         & b_{d }  & a_{d } - \lambda \\
	\end{array} \right ] \\
	&\quad  + (-1)^{d+1} b_1c_d \cdot \det \left[ \begin{array}{ccccc}
		a_2 - \lambda  & c_2     &    \cdots  &  0  \\
		b_3  & a_3 - \lambda    &    \cdots  &  0  \\
		\vdots     &     & \ddots & \vdots\\
		0  &  \cdots  &        b_{d-1 }  & a_{d-1 } - \lambda \\
	\end{array} \right ],
\end{align*}
and  $b = b_1 \cdots b_d$, $c = c_1 \cdots c_d$. Therefore $\det (A(z) - \lambda I) = 0$ if and only if
\begin{equation}\label{3.2:wyzn_el}
	Q(\lambda) = bz + cz^{-1}.
\end{equation}

Let $\E$ be a subset of $\C$ expressed as follows
\begin{equation*}
	\E = \left\{bz + cz^{-1}: z \in \T \right\}.
\end{equation*}
Since the solutions of the equation \eqref{3.2:wyzn_el} belong to the set
\begin{equation*}
Q^{-1}[\E] = \left\{\lambda \in \C: Q(\lambda) \in \E\right\},
\end{equation*}
then $Q^{-1}[\E]$  consists of all eigenvalues of $A(z)$, $z \in \T$ and obviously 
\begin{equation*}
\sigma(\tilde A) = 	\bigcup_{k=1}^d \overline{\lambda_k(\T)} = Q^{-1}[\E].
\end{equation*}

Note that if:
\begin{enumerate}[label={\textup{(\roman*)}}, widest=iii, leftmargin=*]
    \item $b, c \neq 0$ and $|b| \neq |c|$, then the set $\E$ is an ellipse and $\sigma(\tilde A)$ is~a family of  closed curves. Moreover, these curves intersect at a maximum of $d-1$ points;
	\item $b\neq 0 $ and $c = 0$ (or analogously $c \neq 0$ and  $b = 0$), then the set $\E$ is a~circle. The spectrum of the operator $\tilde A$ is a family of closed curves which are intersect at a maximum of $d-1$ points;
	\item  $b, c \neq 0$ and $|b| = |c|$, then the set $\E$ is a line segment, so a spectrum of~the operator $\bA$ is a family of smooth curves. In this case, the curves can intersect at a maximum of $2\left[\frac{d-1}{2}\right]$ points;
	\item $b=0$ and $c=0$, then the spectrum of $\bA$ is a set at most $d$-elements.
\end{enumerate}

Now we give some examples of triangular decompositions and invariant chains under the convolution operators.

\begin{example}
Let $\bA \in \elpd$ be a convolution operator corresponding to a~symbol $A(z)$, $z \in \T$, of the following form
	\begin{equation*}
		A(z) =
		\left[ \begin{array}{ll} 
			z  & 0  \\
			1 & 0 
		\end{array} \right], \quad z \in \T.
	\end{equation*}	
Since $A(z)$  has eigenvalues of the form $\lambda_1(z) = z$ and $\lambda_2(z) = 0$, where $z \in \T$, it follows that $\sigma(A) = \T \cup \left\{0\right\}$.
	
We can apply construction given in Subsection \ref{sub_schur} to build an invariant chain under the operator $A$.  Let $z \in \T$. Using the Gram-Schmidt orthogonalization to the eigenvector corresponding to $\lambda_1(z)$ we get a Schur base for $A(z)$, i.e., a~vectors of the following form 	
	\begin{equation*}
		\varphi^{(1)}(z) = \frac{1}{\sqrt{2}}
		\left[ \begin{array}{l} 
			z    \\
			1  
		\end{array} \right], \quad 
			\varphi^{(2)}(z) = \frac{1}{\sqrt{2}} \left[ \begin{array}{l} 
			-z    \\
			~~1  
		\end{array} \right].
	\end{equation*}	
Then	
	\begin{equation*}
		U(z) = \frac{1}{\sqrt{2}}
		\left[ \begin{array}{ll} 
			z & -z   \\
			1 & ~~1  
		\end{array} \right]
	\end{equation*}
is a unitary matrix and	
	\begin{equation*}
		T(z) = (U(z))^*A(z)U(z) =
		\left[ \begin{array}{ll} 
			z & -z   \\
			0 & ~~0  
		\end{array} \right]
	\end{equation*}
is an upper triangular matrix, unitarily equivalent to matrix $A(z)$. Orthogonal projections on the one-dimensional subspaces of $\C^2$ generated by the Schur vectors $\varphi^{(1)}(z)$ and $\varphi^{(2)}(z)$ are expressed as	
	\begin{align*}
		& 	\Delta P_1(z) = \<\cdot, \varphi^{(1)}(z) \>\varphi^{(1)}(z) = \frac{1}{2}
		\left[ \begin{array}{ll} 
			1 & z   \\
			z^{-1} & 1  
		\end{array} \right], \\
		& \Delta P_2(z) = \<\cdot, \varphi^{(2)}(z) \>\varphi^{(2)}(z) = \frac{1}{2}
		\left[ \begin{array}{ll} 
			~~1 & -z   \\
			-z^{-1} & ~~1  
		\end{array} \right],
	\end{align*}	
respectively. Hence the set
	\begin{equation*}
		\pi(z): \quad  O = P_0(z) < P_1(z) <  P_2(z) = I,
	\end{equation*}
where	
	\begin{equation*}
P_1(z) = \Delta P_1(z), \quad P_2(z) = \Delta P_1(z) + \Delta P_2(z),
	\end{equation*}
is a maximal chain on $\C^2$, invariant under $A(z)$. Moreover, $A(z)$ has the following decomposition
\begin{equation*}
	A(z) = A_0(z) + A_+(z),
\end{equation*}
where
	\begin{align*}
	A_0(z) &  = \Delta P_1(z)A(z)\Delta P_1(z) + \Delta P_1(z)A(z)\Delta P_1(z) = z \cdot \Delta P_1(z) + 0 \cdot\Delta P_2(z), \\  
	A_+(z) & = P_0(z)A(z)\Delta P_1(z) +  P_1(z)A(z)\Delta P_2(z) = -z \<\cdot, \varphi^{(2)}(z) \>\varphi^{(1)}(z).
\end{align*}
The matrices $A_0(z)$ and $A_+(z)$ are respectively a diagonal of $A(z)$, and a nilpotent matrix.

Next, we construct a maximal chain on $\elpd$ that is invariant under the ope\-rator $A$. Assume that  contour $\lambda_1(\T)$ is oriented according to the increase of the argument $z \in \T$, so $\alpha_1 = 1$. Let $\nu \in \Gamma = \lambda_1(\T) \cup \left\{0\right\}$ and consider the matrix-valued function
\begin{equation*}
	P_{\nu}(z) = \chi_{[1, \nu)}(\lambda_1(z))\Delta P_1(z) + \chi_{[1, \nu)}(\lambda_2(z))\Delta P_2(z), \quad z \in \T.
\end{equation*}
Since 
\begin{equation*}
	P_{\nu}(z) = \chi_{[1, \nu)}(\lambda_1(z))\Delta P_1(z) = \chi_{[1, \nu)}(z)\Delta P_1(z), \quad z \in \T,
\end{equation*}
where $\nu \in \lambda_1(\T) = \T$ and 
\begin{align*}
	P_{\nu}(z) =  I, \quad z \in \T,
\end{align*}
where $\nu \in \lambda_2(\T) = \{0\}$, then determining the Fourier coefficients to the function $P_{\nu}(z)$, $z \in \T$, we get the following convolution operators
\begin{align*}
	\bP_{\nu} & = \frac{1}{4\pi} \left[ \begin{array}{cc} 
		-i(\nu -1) & \frac{-i}{2} (\nu^{2} -1) \\
		\arg \nu & -i(\nu -1)  
	\end{array} \right]\, \bS^{-1} + \frac{1}{4\pi} \left[ \begin{array}{cc} 
		\arg \nu & -i(\nu -1) \\
		i(\nu^{-1} -1) & \arg \nu  
	\end{array} \right] \, \bI \\
	& = \frac{1}{4\pi} \left[ \begin{array}{cc} 
		i(\nu^{-1} -1) &  \arg \nu \\
		\frac{-i}{2} (\nu^{-2} -1) & i(\nu^{-1} -1)  
	\end{array} \right] \, \bS \\
	& \quad + \frac{i}{4\pi} \sum_{n \in \Z \setminus \{-1, 0, 1\}} \left[ \begin{array}{cc} 
		\frac{1}{n}(\nu^{-n} -1) & \frac{1}{n-1} (\nu^{-n+1} -1) \\
		\frac{- 1}{n+1}(\nu^{-n-1} - 1) & \frac{1}{n}(\nu^{-n} -1)  
	\end{array} \right] \, \bS^n, \quad \nu \in \T,
\end{align*}
and
\begin{equation*}
	\bP_{\nu = 0} = \bI.
\end{equation*}
A set of projections
\begin{equation*}
	\bP: \quad O = \bP_{\alpha_1} < \bP_{\nu} < \bP_{\mu} < \bP_{\nu = 0} = \bI, \quad \nu, \mu \in \lambda_1(\T),
\end{equation*}
is a chain on $\elpd$, invariant under the operator $\bA$, and reducing for $\bA_0$, where
\begin{align*}
	\bA_{0} & = \frac{1}{2} \left[ \begin{array}{cc} 
		0 & 0 \\
		1 & 0 
	\end{array} \right]\, \bI + \frac{1}{2} \bS + \frac{1}{2} \left[ \begin{array}{cc} 
		0 &  1 \\
		0 & 0  
	\end{array} \right] \, \bS^2. 
\end{align*}
The operator $\bA$ admits a triangular decomposition
\begin{equation*}
	\bA = \bA_0 + \bA_+,
\end{equation*}	
where $\bA_0$ is the diagonal of $A$ with respect to the chain $P$, i.e.
\begin{equation*}
	\bA_0 = \int_{\bP} (d P) \bA (d P), 
\end{equation*}	
and
\begin{align*}
	\bA_{+} & = \frac{1}{2} \left[ \begin{array}{cc} 
		0 & 0 \\
		1 & 0 
	\end{array} \right]\, \bI + \frac{1}{2} \left[ \begin{array}{cc} 
		1 & 0 \\
		0 & -1 
	\end{array} \right] \, \bS + \frac{1}{2} \left[ \begin{array}{cc} 
		0 &  -1 \\
		0 & 0  
	\end{array} \right] \, \bS^2, \quad \bA_{+}^2 = 0.
\end{align*}

\end{example}

\begin{example}
	
Let	a matrix-valued function
	\begin{equation*}
		A(z) =
		\left[ \begin{array}{ll} 
			0 & z  \\
			z & 0 
		\end{array} \right], \quad z \in \T,
	\end{equation*}	
be a symbol of the convolution operator $A \in \elpd$. It is easy to show that eigenvalues of the matrix $A(z)$ are of the form  $\lambda_1(z) = z$ and $\lambda_2(z) = -z$, so the set $\Gamma = \T \cup -\T$ coincides with a spectrum of the operator $\bA$. By repeating construction given in Subsection \ref{sub_schur} we shall build a chain on $\elpd$ for the operator $A$. Assume that both contours $\lambda_1(\T)$ and $\lambda_2(\T)$ are oriented according to the increase of the argument $z \in \T$, so $\alpha_1 = 1$ and $\alpha_2 = -1$, respectively. Then the set of projections
\begin{equation*}
	\bP: \quad O = \bP_{\alpha_1} < \bP_{\nu} < \bP_{\mu} < \bP_{\beta_2} = \bI, \quad \nu \in \T, \ \mu \in - \T,
\end{equation*}
where 
	\begin{equation*}
	\bP_\nu = \frac{\arg \nu}{4 \pi} \left[ \begin{array}{ll} 
		1 & 1   \\
		1 & 1  
	\end{array} \right] \, \bI + \sum_{n \in \Z \setminus \{0\}} \frac{i}{4 n \pi}(\nu^{-n} - 1)\left[ \begin{array}{ll} 
		1 & 1   \\
		1 & 1  
	\end{array} \right] \, \bS^n, \quad \nu \in \T,
\end{equation*}

\begin{align*}
	\bP_\mu & =  \frac{1}{4 \pi} \left[ \begin{array}{cc} 
		\pi + \arg \mu & 3 \pi + \arg \mu   \\
		3 \pi + \arg \mu & \pi + \arg \mu  
	\end{array} \right] \, \bI \\
& \quad + \sum_{n \in \Z \setminus \{0\}} \frac{i}{4 n \pi}((-1)^n \mu^{-n} - 1)\left[ \begin{array}{cc} 
		1 & -1   \\
		-1 & 1  
	\end{array} \right] \, \bS^n, \quad \mu \in -\T,
\end{align*}
and $\beta_2 = e^{3i\pi}$ is a chain on $\elpd$, reducing for the operator $\bA$. Therefore, the operator $A$ is diagonal with respect to the chain $P$, and
\begin{equation*}
	\bA = \int_{\bP} (d P) \bA (d P).
\end{equation*}

\end{example}

\begin{example} \label{Przyklad2} 
	Let $A$ denote a Laurent operator on the space $\elpd$ corresponding to the symbol
	\begin{equation*}
	A(z) = \left[ \begin{array}{cc}
		2(1 + i) & z \\
		4z^{-1} & 2(1 - i)  \\
	\end{array} \right ], \quad z \in \T.
	\end{equation*}
Relying upon our discussion of Subsection \ref{sub_schur}, we see that
	\begin{equation*}
		T(z) =\left[ \begin{array}{cc}
			2 & 5 \\
			0 & 2  \\
		\end{array} \right ], \quad U(z) = \left[ \begin{array}{cc}
			\frac{iz}{\sqrt{5}} &  \frac{2z}{\sqrt{5}}\\
			\frac{2}{\sqrt{5}}  &\frac{i}{\sqrt{5}}  \\
		\end{array} \right ], \quad z \in \T.
	\end{equation*}
Since the symbol of $A$ has in each point $z \in \T$ only one eigenvalue \linebreak $\lambda_1(z) = \lambda_2(z) = 2$, then $\sigma(A) = \{2\}$. But when constructing a chain for A we assume that the spectrum contains two superimposed points $\lambda_1(\T) = \{2\}$, $\lambda_2(\T) = \{2\}$. An invariant chain under the operator $A$ is the set of projections
\begin{equation*}
	\bP: \quad O  < \bP_{\nu=2} < \bI,
\end{equation*}
on $\elpd$, where 
\begin{align*}
	\bP_{\nu =2} & = \frac{1}{5} \left[ \begin{array}{cc} 
		0 & 0 \\
		-2i & 0 
	\end{array} \right]\, \bS^{-1} + \frac{1}{5} \left[ \begin{array}{cc} 
	1 &  0 \\
	0 & 4  
\end{array} \right] \,I + \frac{1}{5}\left[ \begin{array}{cc} 
		0 &  2i \\
		0 & 0  
	\end{array} \right] \, \bS. 
\end{align*}
The operator $\bA$ admits a triangular decomposition
\begin{equation*}
	\bA = \bA_0 + \bA_+,
\end{equation*}	
where $\bA_0$ is the diagonal of $A$ with respect to the chain $P$, i.e.
\begin{equation*}
	\bA_0 = \int_{\bP} (d P) \bA (d P) = 2I, 
\end{equation*}	
and
\begin{align*}
	\bA_{+} & = \left[ \begin{array}{cc} 
		0 & 0 \\
		4 & 0 
	\end{array} \right]\, \bS^{-1} +  \left[ \begin{array}{cc} 
		2i & 0 \\
		0 & -2i 
	\end{array} \right] \, \bI +  \left[ \begin{array}{cc} 
		0 &  1 \\
		0 & 0  
	\end{array} \right] \, \bS, \quad \bA_{+}^2 = 0.
\end{align*}
\end{example}

\begin{example} \label{Przyklad2} 
	Let $A$ denote a convolution operator on the space $\elpd$ corresponding to the symbol
	\begin{equation*}
		A(z) = \left[ \begin{array}{cc}
			z - \frac{1}{2} & \frac{1}{2}i \\
			iz - \frac{1}{2}i  & 2z - \frac{1}{2}  \\
		\end{array} \right ], \quad z \in \T.
	\end{equation*}
It is easy to show that
	\begin{equation*}
		T(z) =\left[ \begin{array}{cc}
			z & z \\
			0 & 2z-1  \\
		\end{array} \right ], \quad U(z) = \left[ \begin{array}{cc}
			\frac{i}{\sqrt{2}} &  -\frac{i}{\sqrt{2}}\\
			\frac{1}{\sqrt{2}}  & \frac{1}{\sqrt{2}} \\
		\end{array} \right ], \quad z \in \T.
	\end{equation*}
	The symbol of $A$ has in each point $z \in \T$ eigenvalues $\lambda_1(z) = z$ and \linebreak $ \lambda_2(z) = 2z -1$, so the spectrum of $A$ is the set consists points on two contours that have one common point: $\sigma(A) = \T \cup \{\mu \in \C: \mu = 2z-1\}$. Note that the symbol of $A$ is not a diagonalizable matrix in only one point $z = 1$.  
	
	Denote that both contours $\lambda_1(\T)$ and $\lambda_2(\T)$ are oriented according to the increase of the argument $z \in \T$, so $\alpha_1 = 1$ and $\alpha_2 = 1$, respectively. Then the set of projections
	\begin{equation*}
		\bP: \quad O = \bP_{\alpha_1} < \bP_{\nu} < \bP_{\mu} < \bP_{\beta_2} = \bI, \quad \nu \in \T, \ \mu \in \lambda_2(\T),
	\end{equation*}
	where 
	\begin{equation*}
		\bP_\nu = \frac{\arg \nu}{2 \pi} \Delta P_1 \, \bI + \sum_{n \in \Z \setminus \{0\}} \frac{i}{2 n \pi}(\nu^{-n} - 1)\Delta P_1 \, \bS^n, \quad \nu \in \T,
	\end{equation*}
	
	\begin{align*}
		\bP_\mu & = \left(\Delta P_1 + \frac{1}{2 \pi} \arg \left(\frac{\mu + 1}{2}\right) \Delta P_2 \right)\, \bI \\
		& \quad + \sum_{n \in \Z \setminus \{0\}} \frac{i}{2 n \pi}\left(\left(\frac{\mu + 1}{2}\right)^{-n} - 1\right)\Delta P_2 \, \bS^n,  \quad \mu \in \lambda_2(\T),
	\end{align*}
and
	\begin{equation*}
	\Delta P_1 = \frac{1}{2} \left[ \begin{array}{ll} 
		1 & i   \\
		-i & 1  
	\end{array} \right], \quad \Delta P_2 = \frac{1}{2} \left[ \begin{array}{ll} 
		1 & -i   \\
		i & 1  
	\end{array} \right], \quad \beta_2 = e^{2i\pi},
\end{equation*}
	is a chain on $\elpd$, invariant under for the operator $\bA$. The operator $\bA$ admits a triangular decomposition
	\begin{equation*}
		\bA = \bA_0 + \bA_+,
	\end{equation*}	
	where $\bA_0$ is the diagonal of $A$ with respect to the chain $P$, i.e.
	\begin{equation*}
		\bA_0 = \int_{\bP} (d P) \bA (d P) = -\Delta P_2 I + (\Delta P_1 + 2 \Delta P_2)S, 
	\end{equation*}	
	and
	\begin{align*}
		\bA_{+} & = \frac{1}{2} \left[ \begin{array}{cc} 
			-1 &  i \\
			i & 1  
		\end{array} \right] \, \bS, \quad \bA_{+}^2 = 0,
	\end{align*}
but there does not exist a spectral measure in the sense of Dunford for the operator $A$.
\end{example}	

\noindent\textbf{Acknowledgements}\\
\textit{The author wishes to express her gratitude to Professor P.A. Cojuhari for formulating the problems and for many useful discussions.}

\bigskip


\noindent Ewelina Zalot\\  
zalot@agh.edu.pl\\
\url{https://orcid.org/0000-0002-1050-6756} \bigskip

\noindent {\small
	\noindent AGH University of Science and Technology\\
	Faculty of Applied Mathematics\\
	al. A. Mickiewicza 30, 30-059 Krak{\'o}w, Poland
}\bigskip

\end{document}